\documentclass[12pt,twoside,reqno]{amsart}
\usepackage[colorlinks=true,citecolor=blue]{hyperref}
\usepackage{mathptmx, amsmath, amssymb, amsfonts, amsthm, mathptmx, enumerate, color,mathrsfs}
\usepackage{graphicx}

\usepackage{epstopdf}

\newtheorem{theorem}{Theorem}[section]

\newtheorem{proposition}[theorem]{Proposition}

\theoremstyle{definition}
\newtheorem{definition}[theorem]{Definition}

\newtheorem{remark}[theorem]{Remark}
\numberwithin{equation}{section}

\begin{document}
\setcounter{page}{1}

\vspace*{2.0cm}
\title[Relationship between nonsmooth VOP and VVI ]
{Interval Valued Vector Variational Inequalities and Vector Optimization Problems via Convexificators}
\author[R.K.Bhardwaj, T. Ram] {Rohit Kumar Bhardwaj$^1$,Tirth Ram$^{1}$}
\maketitle
\vspace*{-0.6cm}

\begin{center}
{\footnotesize

$^1$Department of Mathematics, University of Jammu, Jammu- 180006, India

}\end{center}

\vskip 4mm {\footnotesize \noindent {\bf Abstract.}
 In this paper, we consider interval-valued vector optimization problems $(IVOP)$ and derive their relationships to interval vector variational inequalities $(IVVI)$ of Minty and Stampacchia type in terms of convexificators and LU-efficient solution of $(IVOP)$ using LU-convexity assumption. Furthermore, we consider weak versions of $(IVVI)$ of Minty and Stampacchia type and find the relationships between weak versions of $(IVVI)$ of Minty and Stampacchia type and weakly LU-efficient solution of $(IVOP)$. The results presented in this paper extend and generalized some existing results in the literature.}\\



\renewcommand{\thefootnote}{}
\footnotetext{ $^*$Corresponding author.
\par
E-mail addresses: rohitbhardwaj470@gmail.com (R.K. Bhardwaj), tir1ram2@yahoo.com (T. Ram).
\par
Received February 00, 2022; Accepted February 00, 2022. }

\section{Introduction}
One of the deterministic optimization models that can be used to deal with uncertain data is a problem known as an interval-valued optimization problem. There are three basic ways to describe constrained optimization with uncertainty, which are referred to as the stochastic programming approach, the fuzzy programming approach, and the interval-valued programming approach, see for example \cite{RAH,SLO,WU1}. To find solutions to these problems, numerous techniques have been established. Stochastic and fuzzy optimization problems, on the other hand, are notoriously difficult to resolve. In the method of optimization known as the interval-valued optimization problem, the coefficients of both the objective and constraint functions are represented by closed intervals. As a result, the solution to the stochastic or fuzzy optimization problem will be more difficult to achieve than the solution to the $(IVOP)$. This is the primary reason why the $(IVOP)$ has recently attracted increased interest in the optimization community, see for example \cite{ISH,JEN1,KUM,OSU1,QIA,SIN1,SIN2} and the references contained therein for more information.

In optimization theory, nonsmooth occurrences frequently occur, which has prompted the development of several subdifferential and generalized directional derivative notions. A generalization of plenty of well subdifferentials, particularly Mordukhovich \cite{MOR}, Michel-Penot \cite{MIC}, and Clarke \cite{CLA} subdifferentials is the idea of a convexificator. It has been demonstrated that the idea of convexificators is a helpful tool in the field of nonsmooth optimization. The concept of a convexificator was proposed by Demyanov \cite{DEM} in the year 1994. Convexificators were recently employed by Golestani and Nobakhtian \cite{GOL}, Li and Zhang \cite{LIZ}, Long and Huang \cite{LON} and Luu \cite{LUU} to create the ideal circumstances for nonsmooth optimization problems. We refer to \cite{DEM2,DEM1,KHA,KKL,LII,LUU1,UPA2}, and its sources for further details on convexificators. \\

In 1980, Giannessi \cite{GIA} introduced the notion of vector variational inequalities. Vector variational inequalities have wider applications in optimization, optimal control, economics and equilibrium problems, see, for example \cite{DAF,GIA1,KIN} and the references cited therein. Laha et al. \cite{LAH} define Stampacchia and Minty vector variational inequalities $(VVI)$ in terms of convexificators and use them to identify necessary and sufficient criteria for a point to be a vector minimum point of the $(VOP)$. Mishra and Upadhyay \cite{MIS8} and Upadhyay et al. \cite{UPA3} demonstrated links between nonsmooth $(VOP)$ and $(VVI)$. Zhang et al. \cite{ZHA1,ZHA} investigated optimality conditions for real-valued functions under the assumption of LU-convexity as an extension of convexity. Motivated and inspired by ongoing research work, we adapt the concept of LU-convex function and generalize it to interval-valued vector function. Afterwards, we will use these concepts as a tool to find the relationship between $(IVOP)$ and $(VVI)$ of Stampacchia and Minty types.\\

 The rest sections of this paper are organized as follows: In section 2, some basic definitions and preliminaries are given. In section 3, basic properties and arithmetic for intervals are given. In section 4, we derive relationships among $(IVOP)$, $(IVVI)$ of Minty and Stampacchia type in terms of convexificators and LU-efficient solution of $(IVOP)$ using LU-convexity assumption. Finally we conclude our work in section 5.
	
\section{Preliminaries}

Throughout this paper, let us assume that $\mathbb{R}^{n}$ is the n-dimensional Euclidean space, $\mathbb{R}^{n}_{+}$ be its nonnegative orthant, and $int\mathbb{R}^{n}_{+}$ be the positive orthant of $\mathbb{R}^{n}$. Let the notation $\overline{\mathbb{R}}=\mathbb{R}\cup\{\infty\}$ signify the extended real line, and the notation $\left\langle .,. \right\rangle$ denote the Euclidean inner product. Further, we will assume that $0\neq D\subseteq\mathbb{R}^{n}$ that also contains the Euclidean norm $\|.\|$.\\

The convention for equality and inequalities is as follows:\\
 If $\upsilon,\omega\in \mathbb{R}^{n}$, then\\
$\upsilon\geqq \omega ~\Leftrightarrow~ \upsilon_{j}\geq \omega_{j},~j=1,2,3,...,n$;\\
$\upsilon>\omega ~\Leftrightarrow~ \upsilon_{j}>\omega_{j},~j=1,2,3,...,n$;\\
$\upsilon\geq \omega ~\Leftrightarrow~ \upsilon_{j}\geq \omega_{j},~j=1,2,3,...,n$, but $\upsilon\neq \omega$;\\
$\upsilon\leqq \omega ~\Leftrightarrow~ \upsilon_{j}\leq \omega_{j},~j=1,2,3,...,n$;\\
$\upsilon<\omega ~\Leftrightarrow~ \upsilon_{j}<\omega_{j},~j=1,2,3,...,n$;\\
$\upsilon\leq \omega ~\Leftrightarrow~ \upsilon_{j}\leq \omega_{j},~j=1,2,3,...,n$, but $\upsilon\neq \omega$;\\

First of all, we recall some definitions from \cite{JEY} as follows:
\begin{definition} 
Suppose $\Gamma:D\rightarrow \overline{\mathbb{R}}$ be an extended real valued function, $\upsilon\in D$ and $\Gamma(\upsilon)$ be finite. Then, the {\it lower and upper Dini derivatives} of $\Gamma$ at $\upsilon\in D$ in the direction $\omega\in \mathbb{R}^{n},$ are denoted and defined as follows:
$$\Gamma^{-}(\upsilon,\omega)= \liminf \limits_{\lambda \rightarrow 0}\frac{\Gamma(\upsilon+\lambda \omega)-\Gamma(\upsilon)}{\lambda}.$$
$$\Gamma^{+}(\upsilon,\omega)= \limsup \limits_{\lambda \rightarrow 0}\frac{\Gamma(\upsilon+\lambda \omega)-\Gamma(\upsilon)}{\lambda}.$$
\end{definition}

\begin{definition} 
Suppose $\Gamma:D\rightarrow \overline{\mathbb{R}}$ be an extended real valued function, $\upsilon\in D$ and $\Gamma(\upsilon)$ be finite. Then $\Gamma$ is said to be:
\begin{itemize}
 \item[{(i)}]{\it an upper convexificator} $\partial^{*}\Gamma(\upsilon)\subseteq \mathbb{R}^{n}$ at $\upsilon\in D$, if and only if $\partial^{*}\Gamma(\upsilon)$ is closed and for every $\omega\in \mathbb{R}^{n}$, we have
$$\Gamma^{-}(\upsilon,\omega)\leq \sup \limits_{\zeta \in \partial^{*}\Gamma(\upsilon)} \left\langle \zeta,\omega \right\rangle, $$
 \item[{(ii)}]{\it a lower convexificator} $\partial_{*}\Gamma(\upsilon)\subseteq \mathbb{R}^{n}$ at $\upsilon\in D$, if and only if $\partial_{*}\Gamma(\upsilon)$ is closed and for every $\omega\in \mathbb{R}^{n}$, we have
$$\Gamma^{+}(\upsilon,\omega)\geq \inf \limits_{\zeta \in \partial_{*}\Gamma(\upsilon)} \left\langle \zeta,\omega \right\rangle, $$
 \item[{(iii)}]{\it a convexificator} $\partial^{*}_{*}\Gamma(\upsilon)\subseteq \mathbb{R}^{n}$ at $\upsilon\in D$, if and only if $\partial^{*}_{*}\Gamma(\upsilon)$ is both upper and lower convexificator of $\Gamma$ at $\upsilon$.\\
That is,  for every $\omega\in \mathbb{R}^{n}$, we have
$$\Gamma^{-}(\upsilon,\omega)\leq \sup \limits_{\zeta \in \partial^{*}_{*}\Gamma(\upsilon)} \left\langle \zeta,\omega \right\rangle,~~~\Gamma^{+}(\upsilon,\omega)\geq \inf \limits_{\zeta \in \partial_{*}^{*}\Gamma(\upsilon)} \left\langle \zeta,\omega \right\rangle.$$
\end{itemize}
\end{definition}


\begin{theorem} \cite{JEY} \label{T1}
Suppose $a,b\in D$ and $\Gamma:D\rightarrow \overline{\mathbb{R}}$ be finite and continuous on $(a,b)$. Suppose $\partial^{*}_{*}\Gamma(\omega)$ is a bounded convexificator for all $\omega\in [a,b]$. Then there exists $c\in(a,b)$ such that
$$\Gamma(b)-\Gamma(a)=\left\langle \zeta,b-a \right\rangle,~\mbox{for}~\zeta\in \partial^{*}_{*}\Gamma(\omega).$$ 
\end{theorem}

\begin{definition} \cite{MIS7}
Suppose $\Gamma:X\subseteq \mathbb{R}^{n}\rightarrow \mathbb{R}$ be a function. Then a point $\omega\in X$ is said to be:
\begin{itemize}
\item[{(i)}] efficient (Pareto), iff there exists no $\upsilon\in X$ such that $\Gamma(\upsilon)\leq \Gamma(\omega)$.\\
\item[{(ii)}] weakly efficient (Pareto), iff there exists no $\upsilon\in X$ such that $\Gamma(\upsilon)<\Gamma(\omega)$.
\end{itemize}	
\end{definition}

The notion of convexity for locally Lipschitz vector-valued functions using convexificators is defined as:

\begin{definition}  \cite{LAH}
Suppose $\Gamma=(\Gamma_{1},\Gamma_{2},...,\Gamma_{p}):D\rightarrow \mathbb{R}^{p}$ be a vector-valued function such that $\Gamma_{i}:D\rightarrow \mathbb{R}$ is locally Lipschitz at $\omega\in D$ and admits a bounded convexificator $\partial^{*}_{*}\Gamma(\omega)$ at $\omega$ for all $i\in\ell$. Then $\Gamma$ is said to be:
\begin{itemize}
\item[{(i)}] $\partial^{*}_{*}$-LU-convex at $\omega\in D$ if  
$$\Gamma(\upsilon)-\Gamma(\omega)\geqq \left\langle \zeta,\upsilon-\omega \right\rangle_{p},~\forall~\upsilon\in D,~\zeta\in \partial^{*}_{*} \Gamma(\omega).$$

\item[{(ii)}] strictly $\partial^{*}_{*}$-LU-convex at $\omega_{0}\in D$ if  
$$\Gamma(\upsilon)-\Gamma(\omega) > \left\langle \zeta,\upsilon-\omega \right\rangle_{p},~\forall~\upsilon\in D,~\zeta\in \partial^{*}_{*} \Gamma(\omega).$$
\end{itemize}	
\end{definition}

\section{Interval-valued vector functions}

First, we review several fundamental operations that can be performed at real intervals. For further information on interval analysis, we refer to \cite{MOO1,MOO2}. Let's denote the set of all closed intervals in $\mathbb{R}$ by $\Re$. Suppose $X=[x^{L},x^{U}],~Y=[y^{L},y^{U}]\in \Re$, then the sum and the product are defined by
$$X+Y=\{x+y:x\in X,~y\in Y\}=[x^{L}+y^{L},x^{U}+y^{U}],$$
$$X\times Y=\{xy:x\in X,~y\in Y\}=[min~Q,max~Q],$$
where $Q=\{x^{U}y^{U},x^{U}y^{L},x^{L}y^{U},x^{L}y^{L}\}$. It is important to note that any real number $x$ can be interpreted as the closed interval $X_{x}=[x,x]$, which means that the sum of $x+Y$ is $X_{x}+Y$.\\
Based on the previous procedures, we can describe the product that results from multiplying an interval by a real number $\alpha$ as
$$\alpha X=\{\alpha x:x\in X\}=\left\{\begin{array}{cc}
           [\alpha x^{L},\alpha x^{U}],& \mbox{if}~~\alpha \geq 0,\\
					 \left[\alpha x^{U},\alpha x^{L}\right],& \mbox{if}~~\alpha<0.				
					\end{array}\right.$$			
Note that $-X=\{-x:x\in X\}=[-x^{U},-x^{L}]$. Thus the difference between the two sets will be defined as
$$X-Y=X+(-Y)=[x^{L}-x^{U},x^{U}-x^{L}].$$				
For intervals, an order relation can be defined as	
\begin{itemize}

\item[{(1)}] $X\preceq_{LU} Y~\Longleftrightarrow~x^{L}\leq y^{L}$ and $x^{U}\leq y^{U},$

\item[{(2)}] $X\prec_{LU} Y~\Longleftrightarrow~X\preceq_{LU} Y$ and $X\neq Y,$ that is one of following holds: 
\begin{itemize}
\item[{(a)}] $x^{U}<y^{U}$ and $x^{L}<y^{L}$, or

\item[{(b)}] $x^{U}<y^{U}$ and $x^{L}\leq y^{L}$, or

\item[{(c)}] $x^{U}\leq y^{U}$ and $x^{L}<y^{L}$.
\end{itemize}						

\end{itemize}				
					
\begin{remark}										
Suppose $X=[x^{L},x^{U}],~Y=[y^{L},y^{U}]\in \Re$, then $X$ and $Y$ are said to be comparable if $X\preceq_{LU} Y$ or $X\succeq_{LU} Y$.\\
If any of the following is true, then $X$ and $Y$ cannot be compared to one another:
$$x^{U}>y^{U}~\mbox{and}~x^{L}<y^{L};~ x^{U}\geq y^{U}~\mbox{and}~x^{L}<y^{L};~x^{U}>y^{U}~\mbox{and}~x^{L}\leq y^{L};$$	
$$x^{U}<y^{U}~\mbox{and}~x^{L}>y^{L};~ x^{U}\leq y^{U}~\mbox{and}~x^{L}>y^{L};~x^{U}<y^{U}~\mbox{and}~x^{L}\geq y^{L}.$$	
\end{remark}						

Suppose $X=(X_{1},X_{2},...,X_{n})$ be an interval-valued vector, where every component $X_{k}=[x_{k}^{L},x_{k}^{U}],~k=1,2,...,n$ is a closed interval. We take into consideration two interval-valued vectors denoted by $X=(X_{1},X_{2},...,X_{n})$ and $Y=(Y_{1},Y_{2},...,Y_{n})$ in such a way that $X_{k}$ and $Y_{k}$ are comparable for all $k$ values ranging from $1$ to $n$, then
\begin{itemize}
\item[{(a)}] $X\preceq_{LU} Y$ if $X_{k}\preceq_{LU} Y_{k}$ for all $k=1,2,...,n$,

\item[{(b)}] $X\prec_{LU} Y$ if $X_{k}\preceq_{LU} Y_{k}$ for all $k=1,2,...,n$, and $X_{i}\prec_{LU} Y_{i}$ for at least one $i$.
\end{itemize}			

A function $\Gamma:D\rightarrow \Re$ is said to be interval-valued function if $\Gamma(\omega)=[\Gamma^{L}(\omega),\Gamma^{U}(\omega)]$, where $\Gamma^{L}$ and $\Gamma^{U}$ are real-valued functions defined on $D$ satisfying $\Gamma^{L}(\omega)\leq \Gamma^{U}(\omega)$, for every $\omega\in D$. If $\Gamma_{1},\Gamma_{2},...,\Gamma_{p}:D\rightarrow \Re$ are $p$ interval-valued functions, then we refer to the function $\Gamma=(\Gamma_{1},\Gamma_{2},...,\Gamma_{p}):D\rightarrow \Re^{p}$ an interval-valued vector function. 

\begin{definition} \cite{ZHA}
Suppose $\Gamma=[\Gamma^{L},\Gamma^{U}]:D\rightarrow \Re$ be interval-valued function, then $\Gamma$ is said to be locally Lipschitz at $\omega_{0}\in D$ w.r.t. the Hausdorff metric if there exists $M>0$ and $\delta>0$ such that 
$$d_{H}(\Gamma(\omega),\Gamma(\upsilon))\leq M\|\omega-\upsilon\|,$$
where $d_{H}(\Gamma(\omega),\Gamma(\upsilon))$ is the Hausdorff metric between $\Gamma(\omega)$ and $\Gamma(\upsilon)$, defined by
$$d_{H}(\Gamma(\omega),\Gamma(\upsilon))=max\{|\Gamma(\omega)^{L}-\Gamma(\upsilon)^{L}|,|\Gamma(\omega)^{U}-\Gamma(\upsilon)^{U}|\}.$$
\end{definition}			
If every $\omega_{0}\in D$ is Lipschitz, then $f$ is locally Lipschitz on $D$.

\begin{proposition} \cite{ZHA}
Suppose $\Gamma=[\Gamma^{L},\Gamma^{U}]:D\rightarrow \Re$ be a locally Lipschitz on $D$, then $\Gamma^{L}$ and $\Gamma^{U}$ both are locally Lipschitz on $D$.
\end{proposition}

\begin{definition} 
A function $\Gamma:D\rightarrow \Re$ is said to be $\partial^{*}_{*}$-LU-convex on $D$ if the real valued functions $\Gamma^{L}$ and $\Gamma^{U}$ are $\partial^{*}_{*}$-convex on $D$.
\end{definition}

Suppose	$\Gamma=(\Gamma_{1},\Gamma_{2},...,\Gamma_{p}):D\rightarrow \Re^{p}$ be an interval-valued vector function. Every component function
$\Gamma_{k}=[\Gamma^{L}_{k},\Gamma^{U}_{k}],~k\in\ell=\{1,2,...,p\}$ is a locally Lipschitz interval-valued function defined on $D$. The nonsmooth interval-valued vector optimization problem (in short, $(IVOP)$) is defined as:
$$\mbox{Min}~\left\{\Gamma(\omega)=(\Gamma_{1}(\omega),\Gamma_{2}(\omega),...,\Gamma_{p}(\omega))\right\}~\mbox{such that}~\omega \in D.$$
	
\begin{definition} \cite{ZHA}
A vector $\upsilon\in D$ is said to be :
\begin{itemize}
\item[{(i)}] a LU-efficient solution of the $(IVOP)$ if there exists no $\omega\in D$ such that $\Gamma(\omega)\prec_{LU} \Gamma(\upsilon)$\\
or equivalently 
$$\Gamma_{i}(\omega)\preceq_{LU} \Gamma_{i}(\upsilon),~\forall~i\in\ell,~i\neq j.$$
$$\Gamma_{j}(\omega)\prec_{LU} \Gamma_{j}(\upsilon),~\mbox{for some}~j\in\ell.$$

\item[{(ii)}] a weakly LU-efficient solution of the $(IVOP)$ if there exists no $\omega\in D$ such that 
$$\Gamma_{i}(\mu)\prec_{LU} \Gamma_{i}(\upsilon),~\forall~i\in\ell.$$

\end{itemize}	
\end{definition}
	
\begin{theorem} \cite{LAH} \label{T2}
Suppose $\Gamma=(\Gamma_{1},\Gamma_{2},...,\Gamma_{p}):D\rightarrow \mathbb{R}^{p}$ be a valued vector function such that $\Gamma_{k}:D\rightarrow \Re$ are locally Lipschitz functions at $\upsilon\in D$ and admit bounded convexificators $\partial^{*}_{*}\Gamma(\upsilon),~\forall~k\in\ell$. Then $\Gamma$ is $\partial^{*}_{*}$-convex (strictly) on $D$ iff $\partial^{*}_{*}\Gamma$ is monotone (strictly) on $D$.
\end{theorem}

\section{Interval Valued Minty and Stampacchia Vector Variational Inequalities in terms of convexificators}

An interval-valued vector variational-like inequality problem of Minty type in terms of convexificators (for short, $(IMVVLIP)$) for a nonsmooth case, is to find $\omega\in D$ such that the following cannot hold
$$\left\{\begin{array}{cc}
         \left\langle \zeta^{L},\upsilon-\omega \right\rangle_{p}=\left( \left\langle \zeta^{L}_{1},\upsilon-\omega \right\rangle,\left\langle \zeta^{L}_{2},\upsilon-\omega \right\rangle,...,\left\langle \zeta^{L}_{p},\upsilon-\omega \right\rangle \right)\leqq 0, \\
				  \left\langle \zeta^{U},\upsilon-\omega \right\rangle_{p}=\left( \left\langle \zeta^{U}_{1},\upsilon-\omega \right\rangle,\left\langle \zeta^{U}_{2},\upsilon-\omega \right\rangle,...,\left\langle \zeta^{U}_{p},\upsilon-\omega \right\rangle \right)\leqq 0,			
					\end{array}\right.$$
for all $\upsilon \in D$ and all $\zeta^{L}_{i}\in \partial^{*}_{*}\Gamma_{i}(\omega),~\zeta^{U}_{i}\in \partial^{*}_{*}\Gamma_{i}(\omega),~i\in \ell ={1,2,...,p}$.\\	

An interval-valued vector variational-like inequality problem of Stampacchia type in terms of convexificators (for short, $(ISVVLIP)$) for a nonsmooth case, is to find $\omega\in D$ such that the following cannot hold
$$\left\{\begin{array}{cc}
         \left\langle \xi^{L},\upsilon-\omega \right\rangle_{p}=\left( \left\langle \xi^{L}_{1},\upsilon-\omega \right\rangle,\left\langle \xi^{L}_{2},\upsilon-\omega \right\rangle,...,\left\langle \xi^{L}_{p},\upsilon-\omega \right\rangle \right)\leqq 0, \\
				  \left\langle \xi^{U},\upsilon-\omega \right\rangle_{p}=\left( \left\langle \xi^{U}_{1},\upsilon-\omega \right\rangle,\left\langle \xi^{U}_{2},\upsilon-\omega \right\rangle,...,\left\langle \xi^{U}_{p},\upsilon-\omega \right\rangle \right)\leqq 0,			
					\end{array}\right.$$
for all $\upsilon \in D$ and all $\xi^{L}_{i}\in \partial^{*}_{*}\Gamma_{i}(\upsilon),~\xi^{U}_{i}\in \partial^{*}_{*}\Gamma_{i}(\upsilon),~i\in \ell ={1,2,...,p}$.\\	

We propose essential conditions, which are both necessary and sufficient for an effective solution to the $(IVOP)$.

\begin{theorem}
Suppose $\Gamma=(\Gamma_{1},\Gamma_{2},...,\Gamma_{p}):D\rightarrow \Re^{p}$ be a interval-valued vector function such that $\Gamma_{k}:D\rightarrow \Re$ are locally Lipschitz functions on $D$ and admit bounded convexificators $\partial^{*}_{*}\Gamma(\upsilon)$ for any $\upsilon\in D,~\forall~k\in\ell$. Also, suppose that $\Gamma$ is $\partial^{*}_{*}$-LU-convex on $D$. Then $\omega\in D$ is an LU-efficient solution of $(IVOP)$ iff $\upsilon$ is a solution of $(IMVVLIP)$.
\end{theorem}

\begin{proof}
Suppose on the contrary that $\omega$ is not a solution of $(IMVVLIP)$, then there exist $\upsilon\in D$, $\zeta^{L}_{i}\in \partial^{*}_{*}\Gamma_{i}(\upsilon),~\zeta^{U}_{i}\in \partial^{*}_{*}\Gamma_{i}(\upsilon),~i\in\ell$ such that

\begin{eqnarray} \label{E11}
\left\{\begin{array}{cc}
         \left\langle \zeta^{L},\upsilon-\omega \right\rangle_{p}=\left( \left\langle \zeta^{L}_{1},\upsilon-\omega \right\rangle,\left\langle \zeta^{L}_{2},\upsilon-\omega \right\rangle,...,\left\langle \zeta^{L}_{p},\upsilon-\omega \right\rangle \right) \leqq 0, \\
				  \left\langle \zeta^{U},\upsilon-\omega \right\rangle_{p}=\left( \left\langle \zeta^{U}_{1},\upsilon-\omega \right\rangle,\left\langle \zeta^{U}_{2},\upsilon-\omega \right\rangle,...,\left\langle \zeta^{U}_{p},\upsilon-\omega \right\rangle \right) \leqq 0.			
					\end{array}\right.
\end{eqnarray}

Since each $\Gamma_{i}$ is $\partial^{*}_{*}$-LU-convex. Therefore $\Gamma_{i}^{L}$ and $\Gamma_{i}^{U}$ are $\partial^{*}_{*}$-convex, so we have 
\begin{eqnarray}\label{E12}
\left\{\begin{array}{cc}
          \Gamma_{i}^{L}(\upsilon)-\Gamma_{i}^{L}(\omega) \geqq \left\langle \zeta^{L}_{i},\upsilon-\omega \right\rangle_{p},\\
				  \Gamma_{i}^{U}(\upsilon)-\Gamma_{i}^{U}(\omega) \geqq	\left\langle \zeta^{U}_{i},\upsilon-\omega \right\rangle_{p},		
					\end{array}\right.
\end{eqnarray}					
for all $\omega\in D$ and $i\in \ell$.	
From (\ref{E11}) and (\ref{E12}), it follows that 
$$\Gamma(\upsilon)\prec_{LU} \Gamma(\omega),$$
which is a contradiction.
			
Conversely, suppose that $\omega\in D$ is a solution of $(IMVVLIP)$ but not an efficient soltuion of $(IVOP)$. Then there exists $\upsilon\in D$ such that
\begin{eqnarray}\label{E1}
\Gamma(\upsilon)\prec_{LU} \Gamma(\omega).
\end{eqnarray}	
Using convexity of $D$, take $\mu(t)=\omega+t(\upsilon-\omega)\in D$ for any $t\in[0,1]$. As $\Gamma$ is $\partial^{*}_{*}$-LU-convex on $D$, by Proposition (\ref{T2}), $\forall~t\in[0,1]$, we have
$$\left\{\begin{array}{cc}
         	\Gamma^{L}(\omega+t(\upsilon-\omega))-\Gamma^{L}(\omega)\leqq t [\Gamma^{L}(\upsilon)-\Gamma^{L}(\omega)],\\
					\Gamma^{U}(\omega+t(\upsilon-\omega))-\Gamma^{U}(\omega)\leqq t [\Gamma^{U}(\upsilon)-\Gamma^{U}(\omega)],
					\end{array}\right.$$
or equivalently, for every $i\in \ell$ and $t\in[0,1]$, we have
$$\left\{\begin{array}{cc}
         	\Gamma^{L}_{i}(\omega+t(\upsilon-\omega))-\Gamma^{L}_{i}(\omega)\leq t [\Gamma^{L}_{i}(\upsilon)-\Gamma^{L}_{i}(\omega)],\\
					\Gamma^{U}_{i}(\omega+t(\upsilon-\omega))-\Gamma^{U}_{i}(\omega)\leq t [\Gamma^{U}_{i}(\upsilon)-\Gamma^{U}_{i}(\omega)].
					\end{array}\right.$$
					
By Mean value Theorem (\ref{T1}) on convexificators, for any $i\in\ell$, there exists $\bar{t_{i}}\in(0,t)$ and $\bar{\zeta^{L}_{i}}\in co\partial^{*}_{*}\Gamma_{i}(\mu(\bar{t_{i}}))$, $\bar{\zeta^{U}_{i}}\in co\partial^{*}_{*}\Gamma_{i}(\mu(\bar{t_{i}}))$ such that

$$\left\{\begin{array}{cc}
       \left\langle \bar{\zeta^{L}_{i}},t(\upsilon-\omega) \right\rangle = \Gamma^{L}_{i}(\omega+t(\upsilon-\omega))-\Gamma^{L}_{i}(\omega) ,\\
			 \left\langle \bar{\zeta^{U}_{i}},t(\upsilon-\omega) \right\rangle = \Gamma^{U}_{i}(\omega+t(\upsilon-\omega))-\Gamma^{U}_{i}(\omega),
					\end{array}\right.$$

which implies that for any $i\in \ell$ and for some $\bar{\zeta^{L}_{i}}\in co\partial^{*}_{*}\Gamma_{i}(\mu(\bar{t_{i}}))$, $\bar{\zeta^{U}_{i}}\in co\partial^{*}_{*}\Gamma_{i}(\mu(\bar{t_{i}}))$, we have 

\begin{eqnarray}\label{E2}
\left\{\begin{array}{cc}
       \left\langle \bar{\zeta^{L}_{i}},\upsilon-\omega \right\rangle \leq \Gamma^{L}_{i}(\upsilon)-\Gamma^{L}_{i}(\omega),\\
			 \left\langle \bar{\zeta^{U}_{i}},\upsilon-\omega \right\rangle \leq \Gamma^{U}_{i}(\upsilon)-\Gamma^{U}_{i}(\omega).
					\end{array}\right.
\end{eqnarray} 
 
Suppose $\bar{t_{1}}=\bar{t_{2}}=...=\bar{t_{p}}=\bar{t}$. Multiplying both side of (\ref{E2}) by $\bar{t}$, for $i\in \ell$ and $\bar{\zeta^{L}_{i}}\in co\partial^{*}_{*}\Gamma_{i}(\mu(\bar{t_{i}}))$, $\bar{\zeta^{U}_{i}}\in co\partial^{*}_{*}\Gamma_{i}(\mu(\bar{t_{i}}))$, we have

\begin{eqnarray}\label{E3}
\left\{\begin{array}{cc}
       \left\langle \bar{\zeta^{L}_{i}},\mu(\bar{t})-\omega \right\rangle \leq \bar{t}(\Gamma^{L}_{i}(\upsilon)-\Gamma^{L}_{i}(\omega)),\\
			 \left\langle \bar{\zeta^{U}_{i}},\mu(\bar{t})-\omega \right\rangle \leq \bar{t}(\Gamma^{U}_{i}(\upsilon)-\Gamma^{U}_{i}(\omega)).
					\end{array}\right.
\end{eqnarray} 		
	
Combining	(\ref{E1}) and (\ref{E3}), it follows that $\upsilon$ is not a solution of $(IMVVLIP)$, which is a contradiction.\\
Consider the case when $\bar{t_{1}},\bar{t_{2}},...,\bar{t_{p}}$ are not all equal. Suppose $\bar{t_{1}}\neq \bar{t_{2}}$. Then from (\ref{E2}), we have 

$$\left\{\begin{array}{cc}
       \left\langle \bar{\zeta^{L}_{1}},\upsilon-\omega \right\rangle \leq \Gamma^{L}_{1}(\upsilon)-\Gamma^{L}_{1}(\omega),\\
			 \left\langle \bar{\zeta^{U}_{1}},\upsilon-\omega \right\rangle \leq \Gamma^{U}_{1}(\upsilon)-\Gamma^{U}_{1}(\omega),
					\end{array}\right.$$
for some $\bar{\zeta^{L}_{1}}\in co\partial^{*}_{*}\Gamma_{1}(\mu(\bar{t_{1}}))$, $\bar{\zeta^{U}_{1}}\in co\partial^{*}_{*}\Gamma_{1}(\mu(\bar{t_{1}}))$						
and 
$$\left\{\begin{array}{cc}
       \left\langle \bar{\zeta^{L}_{2}},\upsilon-\omega \right\rangle \leq \Gamma^{L}_{2}(\upsilon)-\Gamma^{L}_{2}(\omega),\\
			 \left\langle \bar{\zeta^{U}_{2}},\upsilon-\omega \right\rangle \leq \Gamma^{U}_{2}(\upsilon)-\Gamma^{U}_{2}(\omega),
					\end{array}\right.$$
					
for some $\bar{\zeta^{L}_{2}}\in co\partial^{*}_{*}\Gamma_{2}(\mu(\bar{t_{2}}))$, $\bar{\zeta^{U}_{2}}\in co\partial^{*}_{*}\Gamma_{2}(\mu(\bar{t_{2}}))$.\\					
Since $\Gamma_{1}$ and $\Gamma_{2}$ are $\partial^{*}_{*}$-LU-convex. Therefore $\Gamma_{1}^{L}$ and $\Gamma_{2}^{U}$ are $\partial^{*}_{*}$-convex, so by Proposition (\ref{T2}),  we have 

$$\left\{\begin{array}{cc}
       \left\langle \bar{\zeta^{L}_{1}}-\bar{\zeta^{L}_{12}},\mu(\bar{t_{1}})-\mu(\bar{t_{2}}) \right\rangle \geq 0,~\forall~\bar{\zeta^{L}_{12}}\in co\partial^{*}_{*}\Gamma_{1}(\mu(\bar{t_{2}})),\\
			 \left\langle \bar{\zeta^{U}_{1}}-\bar{\zeta^{U}_{12}},\mu(\bar{t_{1}})-\mu(\bar{t_{2}}) \right\rangle \geq 0,~\forall~\bar{\zeta^{U}_{12}}\in co\partial^{*}_{*}\Gamma_{1}(\mu(\bar{t_{2}})),
					\end{array}\right.$$
and 

$$\left\{\begin{array}{cc}
       \left\langle \bar{\zeta^{L}_{2}}-\bar{\zeta^{L}_{21}},\mu(\bar{t_{2}})-\mu(\bar{t_{1}}) \right\rangle \geq 0,~\forall~\bar{\zeta^{L}_{21}}\in co\partial^{*}_{*}\Gamma_{2}(\mu(\bar{t_{1}})),\\
			 \left\langle \bar{\zeta^{U}_{2}}-\bar{\zeta^{U}_{21}},\mu(\bar{t_{2}})-\mu(\bar{t_{1}}) \right\rangle \geq 0,~\forall~\bar{\zeta^{U}_{21}}\in co\partial^{*}_{*}\Gamma_{2}(\mu(\bar{t_{1}})).
					\end{array}\right.$$

If $\bar{t_{1}}-\bar{t_{2}}>0$, then
\begin{eqnarray*}
 \left\{\begin{array}{cc}
       \left\langle \bar{\zeta^{L}_{12}},\upsilon-\omega \right\rangle \leq \Gamma^{L}_{1}(\upsilon)-\Gamma^{L}_{1}(\omega),\\
			 \left\langle \bar{\zeta^{U}_{12}},\upsilon-\omega \right\rangle \leq \Gamma^{U}_{1}(\upsilon)-\Gamma^{U}_{1}(\omega).
					\end{array}\right.
\end{eqnarray*}

If $\bar{t_{2}}-\bar{t_{1}}>0$, then	
\begin{eqnarray*}
\left\{\begin{array}{cc}
       \left\langle \bar{\zeta^{L}_{21}},\upsilon-\omega \right\rangle \leq \Gamma^{L}_{2}(\upsilon)-\Gamma^{L}_{2}(\omega),\\
			 \left\langle \bar{\zeta^{U}_{21}},\upsilon-\omega \right\rangle \leq \Gamma^{U}_{2}(\upsilon)-\Gamma^{U}_{2}(\omega).
					\end{array}\right.
\end{eqnarray*}
	
For $\bar{t_{1}}\neq \bar{t_{2}}$, set $\bar{t}=min\{\bar{t_{1}},\bar{t_{2}}\}$, there exists $\bar{\zeta^{L}_{i}}\in co\partial^{*}_{*}\Gamma_{i}^{L}(\mu(\bar{t}))$, $\bar{\zeta^{L}_{i}}\in co\partial^{*}_{*}\Gamma_{i}^{U}(\mu(\bar{t}))$, for any $i=1,2$ such that
\begin{eqnarray*}
\left\{\begin{array}{cc}
       \left\langle \bar{\zeta^{L}_{i}},\upsilon-\omega \right\rangle \leq \Gamma^{L}_{i}(\upsilon)-\Gamma^{L}_{i}(\omega),\\
			 \left\langle \bar{\zeta^{U}_{i}},\upsilon-\omega \right\rangle \leq \Gamma^{U}_{i}(\upsilon)-\Gamma^{U}_{i}(\omega).
					\end{array}\right.
\end{eqnarray*}

Continuation this process, we can find $\hat{t}\in (0,t)$ such that $\hat{t}=min\{\bar{t_{1}},\bar{t_{2}},...,\bar{t_{p}}\}$ and $\hat{\zeta^{L}_{i}}\in co\partial^{*}_{*}\Gamma_{i}^{L}(\mu(\hat{t}))$, $\hat{\zeta^{U}_{i}}\in co\partial^{*}_{*}\Gamma_{i}^{U}(\mu(\hat{t})),~\forall~i\in \ell$ such that

\begin{eqnarray*}
\left\{\begin{array}{cc}
       \left\langle \hat{\zeta^{L}_{i}},\upsilon-\omega \right\rangle \leq \Gamma^{L}_{i}(\upsilon)-\Gamma^{L}_{i}(\omega),\\
			 \left\langle \hat{\zeta^{U}_{i}},\upsilon-\omega \right\rangle \leq \Gamma^{U}_{i}(\upsilon)-\Gamma^{U}_{i}(\omega).
					\end{array}\right.
\end{eqnarray*}
Multiplying the above inequalities by $\hat{t}$, we have

\begin{eqnarray*}
\left\{\begin{array}{cc}
       \left\langle \hat{\zeta^{L}_{i}},\mu(\hat{t})-\omega \right\rangle \leq \hat{t}(\Gamma^{L}_{i}(\upsilon)-\Gamma^{L}_{i}(\omega)),\\
			 \left\langle \hat{\zeta^{U}_{i}},\mu(\hat{t})-\omega \right\rangle \leq \hat{t}(\Gamma^{U}_{i}(\upsilon)-\Gamma^{U}_{i}(\omega)).
					\end{array}\right.
\end{eqnarray*}

From (\ref{E1}), for some $\mu(\hat{t})\in D$, we have
\begin{eqnarray*}
\left\{\begin{array}{cc}
       \left\langle \hat{\zeta^{L}},\mu(\hat{t})-\omega \right\rangle_{p} \leqq 0,\\
			 \left\langle \hat{\zeta^{U}},\mu(\hat{t})-\omega \right\rangle_{p} \leqq 0,
					\end{array}\right.
\end{eqnarray*}
which is a contradiction. This complete the proof.													
\end{proof}

\begin{theorem}
Suppose $\Gamma=(\Gamma_{1},\Gamma_{2},...,\Gamma_{p}):D\rightarrow \Re^{p}$ be a interval-valued vector function such that $\Gamma_{k}:D\rightarrow \Re$ are locally Lipschitz functions on $D$ and admit bounded convexificators $\partial^{*}_{*}\Gamma(\upsilon)$ for any $\upsilon\in D,~\forall~k\in\ell$. Also, suppose that $\Gamma$ is $\partial^{*}_{*}$-LU-convex on $D$. If $\omega\in D$ is a solution of $(ISVVLIP)$, then $\omega$ is a solution of $(IMVVLIP)$.
\end{theorem}

\begin{proof}
Suppose $\omega\in D$ is a solution of $(ISVVLIP)$, then for any $\upsilon\in D$, $\xi^{L}\in \partial^{*}_{*}\Gamma(\upsilon),~\xi^{U}\in \partial^{*}_{*}\Gamma(\upsilon)$ the following cannot hold
\begin{eqnarray*} 
\left\{\begin{array}{cc}
         \left\langle \xi^{L},\upsilon-\omega \right\rangle_{p} \leqq 0, \\
				 \left\langle \xi^{U},\upsilon-\omega \right\rangle_{p} \leqq 0.			
					\end{array}\right.
\end{eqnarray*}

Since $\Gamma$ is $\partial^{*}_{*}$-LU-convex. Therefore $\Gamma^{L}$ and $\Gamma^{U}$ are $\partial^{*}_{*}$-convex, so by Theorem (\ref{T2}), $\partial^{*}_{*}\Gamma^{L}$ and $\partial^{*}_{*}\Gamma^{U}$ are monotone over D, which implies that, for any $\upsilon\in D$ and $\zeta^{L}\in \partial^{*}_{*}\Gamma(\omega),~\zeta^{U}\in \partial^{*}_{*}\Gamma(\omega)$ the following cannot hold

\begin{eqnarray*} 
\left\{\begin{array}{cc}
         \left\langle \zeta^{L},\upsilon-\omega \right\rangle_{p} \leqq 0,  \\
				 \left\langle \zeta^{U},\upsilon-\omega \right\rangle_{p} \leqq 0 .			
					\end{array}\right.
\end{eqnarray*}
Hence $\upsilon\in D$ is a solution of $(IMVVLIP)$.
\end{proof}


An interval-valued weak vector variational-like inequality problem of Minty type in terms of convexificators (for short, $(IWMVVLIP)$) for a nonsmooth case, is to find $\omega\in D$ such that the following cannot hold
$$\left\{\begin{array}{cc}
         \left\langle \zeta^{L},\upsilon-\omega \right\rangle_{p}=\left( \left\langle \zeta^{L}_{1},\upsilon-\omega \right\rangle,\left\langle \zeta^{L}_{2},\upsilon-\omega \right\rangle,...,\left\langle \zeta^{L}_{p},\upsilon-\omega \right\rangle \right) <0, \\
				  \left\langle \zeta^{U},\upsilon-\omega \right\rangle_{p}=\left( \left\langle \zeta^{U}_{1},\upsilon-\omega \right\rangle,\left\langle \zeta^{U}_{2},\upsilon-\omega \right\rangle,...,\left\langle \zeta^{U}_{p},\upsilon-\omega \right\rangle \right) <0,			
					\end{array}\right.$$
for all $\upsilon \in D$ and all $\zeta^{L}_{i}\in \partial^{*}_{*}\Gamma_{i}(\omega),~\zeta^{U}_{i}\in \partial^{*}_{*}\Gamma_{i}(\omega),~i\in \ell ={1,2,...,p}$.\\	

An interval-valued weak vector variational-like inequality problem of Stampacchia type in terms of convexificators (for short, $(IWSVVLIP)$) for a nonsmooth case, is to find $\omega\in D$ such that the following cannot hold
$$\left\{\begin{array}{cc}
         \left\langle \xi^{L},\upsilon-\omega \right\rangle_{p}=\left( \left\langle \xi^{L}_{1},\upsilon-\omega \right\rangle,\left\langle \xi^{L}_{2},\upsilon-\omega \right\rangle,...,\left\langle \xi^{L}_{p},\upsilon-\omega \right\rangle \right)  <0, \\
				  \left\langle \xi^{U},\upsilon-\omega \right\rangle_{p}=\left( \left\langle \xi^{U}_{1},\upsilon-\omega \right\rangle,\left\langle \xi^{U}_{2},\upsilon-\omega \right\rangle,...,\left\langle \xi^{U}_{p},\upsilon-\omega \right\rangle \right) <0,			
					\end{array}\right.$$
for all $\upsilon \in D$ and all $\xi^{L}_{i}\in \partial^{*}_{*}\Gamma_{i}(\upsilon),~\xi^{U}_{i}\in \partial^{*}_{*}\Gamma_{i}(\upsilon),~i\in \ell ={1,2,...,p}$.\\	

\begin{theorem} \label{T3}
Suppose $\Gamma=(\Gamma_{1},\Gamma_{2},...,\Gamma_{p}):D\rightarrow \Re^{p}$ be a interval-valued vector function such that $\Gamma_{k}:D\rightarrow \Re$ are locally Lipschitz functions at $\upsilon\in D$ and admit bounded convexificators $\partial^{*}_{*}\Gamma(\upsilon),~\forall~k\in\ell$. Also, suppose that $\Gamma$ is $\partial^{*}_{*}$-LU-convex on $D$. Then $\omega\in D$ is a weak  efficient solution of $(IVOP)$ iff $\omega$ is a solution of $(IWSVVLIP)$.
\end{theorem}

\begin{proof}
Suppose that $\omega$ is a weak efficient solution of $(IVOP)$. Then there exists no $\upsilon\in D$ such that
$$\Gamma_{i}(\upsilon)\prec_{LU} \Gamma_{i}(\omega),~\forall~i\in\ell.$$
Thus there exists no $\upsilon\in D$ such that
\begin{eqnarray*} 
\left\{\begin{array}{cc}
       (\Gamma^{L}_{1}(\upsilon)-\Gamma^{L}_{1}(\omega),\Gamma^{L}_{2}(\upsilon)-\Gamma^{L}_{2}(\omega),...,\Gamma^{L}_{p}(\upsilon)-\Gamma^{L}_{p}(\omega)) <0,\\
			 (\Gamma^{U}_{1}(\upsilon)-\Gamma^{U}_{1}(\omega),\Gamma^{U}_{2}(\upsilon)-\Gamma^{U}_{2}(\omega),...,\Gamma^{U}_{p}(\upsilon)-\Gamma^{U}_{p}(\omega)) <0.			
					\end{array}\right.
\end{eqnarray*}

Using convexity of $D$, $\omega+t(\upsilon-\omega)\in D$, for any $t\in[0,1]$, which implies that
$$\frac{f(\omega+t(\upsilon-\omega))}{t}<0,~\mbox{for}~t\in[0,1].$$
Taking limit inf as $t\to 0$, we have
\begin{eqnarray*} 
\left\{\begin{array}{cc}
         (\Gamma^{L-}_{1}(\omega,\upsilon-\omega),\Gamma^{L-}_{2}(\omega,\upsilon-\omega),...,\Gamma^{L-}_{p}(\omega,\upsilon-\omega)) <0,\\
				 (\Gamma^{U-}_{1}(\omega,\upsilon-\omega),\Gamma^{U-}_{2}(\omega,\upsilon-\omega),...,\Gamma^{U-}_{p}(\omega,\upsilon-\omega)) <0.		
					\end{array}\right.
\end{eqnarray*}
Since $\Gamma_{i}$ admit bounded convexificators $\partial^{*}_{*}\Gamma_{i}(\upsilon)$, for any $i\in \ell$, there exists no $\upsilon\in D$ such that 
$$\left\{\begin{array}{cc}
         \left\langle \xi^{L},\upsilon-\omega \right\rangle_{p} <0,\mbox{for all}~\xi^{L}_{i}\in \partial^{*}_{*}\Gamma_{i}(\upsilon),\\
				 \left\langle \xi^{U},\upsilon-\omega \right\rangle_{p} <0,\mbox{for all}~\xi^{U}_{i}\in \partial^{*}_{*}\Gamma_{i}(\upsilon).			
					\end{array}\right.$$
Hence $\upsilon$ is a solution of $(IWSVVLIP)$.\\
Conversely, suppose that $\upsilon$ is not a weak efficient solution of $(IVOP)$. Then there exists $\omega\in D$ such that
$$\left\{\begin{array}{cc}
         \Gamma^{L}(\omega)-\Gamma^{L}(\upsilon) <0, \\
				 \Gamma^{U}(\omega)-\Gamma^{U}(\upsilon) <0.			
					\end{array}\right.$$					
Using $\partial^{*}_{*}$-LU-convex of $\Gamma$ at $\upsilon$, there exists $\omega\in D$ such that
$$\left\{\begin{array}{cc}
         \left\langle \xi^{L},\upsilon-\omega \right\rangle_{p} <0,~\mbox{for all}~ \xi^{L}_{i}\in \partial^{*}_{*}\Gamma_{i}(\upsilon),\\
				 \left\langle \xi^{U},\upsilon-\omega \right\rangle_{p} <0,~\mbox{for all}~ \xi^{U}_{i}\in \partial^{*}_{*}\Gamma_{i}(\upsilon),			
					\end{array}\right.$$
which is a contradiction. Hence the proof.
\end{proof}

\begin{theorem}
Suppose $\Gamma=(\Gamma_{1},\Gamma_{2},...,\Gamma_{p}):D\rightarrow \Re^{p}$ be a interval-valued vector function such that $\Gamma_{k}:D\rightarrow \Re$ are locally Lipschitz functions at $\upsilon\in D$ and admit bounded convexificators $\partial^{*}_{*}\Gamma(\upsilon),~\forall~k\in\ell$. Also, suppose that $\Gamma$ is $\partial^{*}_{*}$-LU-convex on $D$. Then $\omega\in D$ is a solution of $(IWMVVLIP)$ iff $\omega$ is a solution of $(IWSVVLIP)$.
\end{theorem}

\begin{proof}
Suppose that $\omega$ is a solution of $(IWMVVLIP)$. Consider any sequence $\{t_{l}\}$ with $t_{l}\in (0,1]$ such that $t_{l}\to 0$ as $l\to \infty$. As $D$ is convex, we have 
$\omega_{l}:=\omega+t_{l}(\upsilon-\omega)\in D,~\mbox{for all}~\upsilon\in D.$
Since $\omega$ is a solution of $(IWMVVLIP)$, for $\xi^{L}_{p}\in \partial^{*}_{*}\Gamma_{p}(\omega_{l})$, $\xi^{U}_{p}\in \partial^{*}_{*}\Gamma_{p}(\omega_{l})$ there exists no $\upsilon\in D$ such that
$$\left\{\begin{array}{cc}
         \left\langle \xi^{L}_{p},\omega_{l}-\upsilon \right\rangle_{p} <0,\\
				 \left\langle \xi^{U}_{p},\omega_{l}-\upsilon \right\rangle_{p} <0.		
					\end{array}\right.$$
Since each $\Gamma_{i}$ is locally Lipschitz and admit bounded convexificators on D, there exists $k>0$ such that $\|\xi_{l_{i}}\|\leq k$, which means that the sequence ${\xi_{l_{i}}}$	converges to $\xi_{i}$ for all $i\in \ell$. Also the convexificators $\partial^{*}_{*}\Gamma_{i}(\upsilon)$ are closed for all $i\in \ell$ and $\upsilon\in D$, it follows that $\omega_{l}\to\omega$	and $\xi_{l_{i}}\to \bar{\xi_{i}}$ as $l\to \infty$ with $\bar{\xi_{i}}\in \partial^{*}_{*}\Gamma_{i}(\upsilon)$ for all $i\in \ell$.	Thus for $\xi^{L}\in \partial^{*}_{*}\Gamma(\upsilon)$ and $\xi^{U}\in \partial^{*}_{*}\Gamma(\upsilon)$ there exists no $\upsilon\in D$ such that 	
$$\left\{\begin{array}{cc}
         \left\langle \xi^{L},\upsilon-\omega \right\rangle_{p} <0,\\
				 \left\langle \xi^{U},\upsilon-\omega \right\rangle_{p} <0.			
					\end{array}\right.$$
Hence $\omega$ is a solution of $(IWSVVLIP)$.\\
Conversely, suppose that $\omega$ is a solution of $(IWSVVLIP)$. Then, for any $\upsilon\in D$ and $\xi^{L}\in \partial^{*}_{*}\Gamma(\upsilon)$, $\xi^{U}\in \partial^{*}_{*}\Gamma(\upsilon)$ the following cannot hold
$$\left\{\begin{array}{cc}
         \left\langle \xi^{L},\upsilon-\omega \right\rangle_{p} <0,\\
				 \left\langle \xi^{U},\upsilon-\omega \right\rangle_{p} <0.			
					\end{array}\right.$$
Since $\Gamma$ is $\partial^{*}_{*}$-LU-convex on $D$. Therefore $\Gamma^{L}$ and $\Gamma^{U}$ are $\partial^{*}_{*}$-convex, so by Theorem (\ref{T2}) $\partial^{*}_{*}\Gamma^{L}$ and $\partial^{*}_{*}\Gamma^{U}$ are monotone over D, which implies that, for any $\upsilon\in D$ and $\xi^{L}\in \partial^{*}_{*}\Gamma(\omega)$, $\xi^{U}\in \partial^{*}_{*}\Gamma(\omega)$ the following cannot hold 	
$$\left\{\begin{array}{cc}
         \left\langle \xi^{L},\upsilon-\omega \right\rangle_{p} <0,\\
				 \left\langle \xi^{U},\upsilon-\omega \right\rangle_{p} <0.			
					\end{array}\right.$$
Hence $\omega$ is a solution of $(IWMVVLIP)$.												
\end{proof}
\begin{theorem}
Suppose $\Gamma=(\Gamma_{1},\Gamma_{2},...,\Gamma_{p}):D\rightarrow \Re^{p}$ be a interval-valued vector function such that $\Gamma_{k}:D\rightarrow \Re$ are locally Lipschitz functions at $\upsilon\in D$ and admit bounded convexificators $\partial^{*}_{*}\Gamma(\upsilon),~\forall~k\in\ell$. Also, suppose that $\Gamma$ is strictly $\partial^{*}_{*}$-LU-convex on $D$. Then $\omega\in D$ is an efficient solution of $(IVOP)$ iff $\omega$ is a weak LU-efficient solution of $(IVOP)$.
\end{theorem}

\begin{proof}
Every efficient solution is a weak efficient solution of $(IVOP)$.\\
Conversely, suppose that $\omega$ is a weak LU-efficient solution of $(IVOP)$, but not a LU-efficient solution of $(IVOP)$. Then there exists $\upsilon\in D$ such that 
$$\Gamma(\upsilon)\prec_{LU} \Gamma(\omega).$$
				  
Since $\Gamma$ is strictly $\partial^{*}_{*}$-LU-convex on $D$. Therefore $\Gamma^{L}$ and $\Gamma^{U}$ are $\partial^{*}_{*}$ strictly convex, so for any $\xi^{L}\in \partial^{*}_{*}\Gamma(\omega)$ and  $\xi^{U}\in \partial^{*}_{*}\Gamma(\omega)$, there exist $\omega\in D$ 	
$$\left\{\begin{array}{cc}
         \left\langle \xi^{L},\upsilon-\omega \right\rangle_{p} <0,\\
				 \left\langle \xi^{U},\upsilon-\omega \right\rangle_{p} <0,			
					\end{array}\right.$$	
which is not a solution of $(IWMVVLIP)$. By Theorem (\ref{T3}), $\upsilon$ is not a weak efficient solution of $(IVOP)$, which is a contradiction. Hence the proof. 									
\end{proof}

\section{Conclusion}
In this paper, we have considered a class of nonsmooth $(IVOP)$ and Stampacchia and Minty type vector variational inequalities in terms of convexificators which are weaker version of the notion of subdifferentials. We have established the relationships among the solutions of the Stampacchia and Minty type vector variational inequalities and LU-efficient solution of $(IVOP)$ involving locally Lipschitz functions. In addition, we have also established the relationships between weak LU-efficient solution of $(IVOP)$ and weak versions of the Stampacchia and Minty vector variational inequalities.

\end{document}